\documentclass[11pt]{article}
\usepackage{graphicx}
\usepackage{mathrsfs}
\usepackage{xypic}

\usepackage{color}
\usepackage{amsfonts}
\usepackage{amsmath}
\usepackage{amssymb}
\usepackage{bm}
\usepackage{CJK}
\newtheorem{theorem}{Theorem}[section]
\newtheorem{lemma}[theorem]{Lemma}
\newtheorem{definition}[theorem]{Definition}

\newtheorem{conjecture}[theorem]{Conjecture}
\newtheorem{problem}[theorem]{Problem}
\numberwithin{equation}{section}

\newcommand{\ba}{\begin{array}}
	\newcommand{\ea}{\end{array}}
\newcommand{\bt}{\begin{tabular}}
	\newcommand{\et}{\end{tabular}}
\newcommand{\btb}{\begin{table}}
	\newcommand{\etb}{\end{table}}
\newcommand{\bc}{\begin{center}}
	\newcommand{\ec}{\end{center}}
\newcommand{\bea}{\begin{eqnarray}}
	\newcommand{\eea}{\end{eqnarray}}
\newcommand{\Bea}{\begin{eqnarray*}}
	\newcommand{\Eea}{\end{eqnarray*}}
\newcommand{\beq}{\begin{equation}}
	\newcommand{\eeq}{\end{equation}}

\begin{document}
	\baselineskip 16.5pt

	\title{Studying the  divisibility of power LCM matrics by power GCD matrices on  gcd-closed sets
	}
	\author{
		Jianrong Zhao$^{a,b,1}$ 
		\thanks{J.R. Zhao  is the corresponding author and was  supported by National Science Foundation of China Grant \#12371333.}
		\thanks{$^1$ E-Mail: mathzjr@swufe.edu.cn. mathzjr@foxmail.com $^2$ E-mail: wcx822@foxmail.com. $^3$ E-mail: fuyu17761301719@163.com}
		,~Chenxu Wang$^{a,2}$, Yu Fu$^{c,3}$
	}
	\date{}
	\maketitle
	\begin{center}
		$^a$ School of  Mathematics,	
		Southwestern University of Finance and Economics,
		
		Chengdu 611130, P. R. China.
		
		$^b$ Big Data Laboratory on Financial Security and Behavior, SWUFE (Laboratory of Philosophy and Social Sciences, Ministry of Education), Chengdu 611130, China

		$^c$ Department of Mathematics, Sichuan University,
		
		Chengdu 610064, P. R. China.
		
	\end{center}
	
	\noindent{\bf Abstract.}
	Let  $S=\{x_1,\ldots, x_n\}$ be a gcd-closed set (i.e. $(x_i,x_j)\in S $ for all $1\le i,j\le n$). In 2002, Hong proposed the divisibility problem of characterizing all gcd-closed sets $S$ with $|S|\ge 4$ such that  the GCD matrix $(S)$ divides the LCM matrix $[S]$ in the ring $M_{n}(\mathbb{Z})$.
	For $x\in S,$ let $G_S(x):=\{z\in S:  z<x, z|x \text{ and } (z|y|x, y\in S)\Rightarrow y\in\{z,x\}\}$. In 2009, Feng, Hong and Zhao answered this problem in the context where $\max_{x \in S}\{|G_S(x)|\} \leq 2$. In 2022, Zhao, Chen and Hong obtained a necessary and sufficient condition on the gcd-closed set $S$ with $\max_{x \in S}\{|G_S(x)|\}=3$ such that $(S)|\left[S\right].$  Meanwhile, they raised a conjecture on the necessary and sufficient condition  such that $(S)|\left[S\right]$ holds  for the remaining case $\max_{x \in S}\{|G_S(x)|\}\ge 4$.
	In this papar, we confirm the Zhao-Chen-Hong conjecture from a novel perspective,  consequently solve Hong's open problem completely.
	\vspace{1mm}
	
	\noindent \textbf{Keywords:} GCD matrix,  LCM matrix, Gcd-closed set,
	Greatest-type divisor, Divisibility, M\"obius function
	\vspace{1mm}
	
	\noindent {\small \textbf{MR(2000) Subject Classification:} Primary
		11B73, 11A07, 15C36}
	
	
	\section{Introduction}
	In 1875, Smith\cite{S1875} presented his celebrated result, which states that the determinant of the matrix $((i, j))$ is given by $\operatorname{det}((i, j)) = \prod_{k=1}^{n} \varphi(k)$, where the $n \times n$ matrix $((i, j))$ has its $ij$-entry defined as the greatest common divisor  of $i$ and $j$, and $\varphi$ denotes Euler's totient function. In the same paper, Smith also established that the determinant of the $n \times n $ matrix $ (f(i, j)) $, where each entry is defined as $ f $ evaluated at the greatest common divisor $ ij $ of indices $ i $ and $ j $, is given by 	
	$
	\prod_{k=1}^{n}(f * \mu)(k),
	$	
	where $ f * \mu $ denotes the Dirichlet convolution of the arithmetic function $ f $ with the M\"obius function $\mu$. 
	Subsequently, the results of Smith, as compiled by Dickson in \cite{D71B}, garnered significant attention from scholars and led to several intriguing generalizations, see e.g. \cite{A72PM,BL89AMC,BL93JNT,CN02BUM,H02AA,L88LMA,M86CMB,HHL16AM,HL18PM}. 
	
	Let $S=\{x_1,\ldots, x_n\}$ be a set of $n$ distinct positive integers.
	Given any real number $e\ge 0$, denote by $((x_i, x_j)^e)$ (abbreviated by $(S^e)$ ) the  $n \times n
	$  matrix  having  the $e$-th power of  greatest common divisor  $(x_i, x_j)$ of  $x_i$  and  $x_j$  as its  $i j$ -entry matrix,	
	and $([x_i, x_j]^e)$ (abbreviated by $[S^e]$) the  $n \times n
	$  matrix  having  the $e$-th power of   least common multiple  $[x_i, x_j]$ of  $x_i$  and  $x_j$  as its  $ij$-entry matrix.
	We say that $(S^e)$ and $[S^e]$ the {\it  power greatest  common divisor (GCD) matrix} and the {\it power least common multiple (LCM)	 matrix}, respectively.	 
	If \( e = 1 \), we refer to \( (S) \) and \( [S] \) as the {\it greatest common divisor (GCD) matrix} and the {\it least common multiple (LCM) matrix}, respectively.
	In 1989, Beslin and Ligh \cite{BL89LAA} demonstrated that the GCD matrix $(S)$ on any set $S$ of positive integers is always nonsingular. Furthermore, in 1993, Bourque and Ligh \cite{BL93LMA} proved that the power GCD matrix $(S^e)$ on any set $S$ of positive integers is also nonsingular. However, it is difficult to characterize the nonsingularity of the power LCM matrix $[S^e]$ on any set $S$ of positive integers. Numerous findings regarding the nonsingularity of the power LCM matrix defined on specific sets have been obtained, see e.g. \cite{AST05LMA,BL92LAA,C07CMJ,HMM20JCTA,H99JA,HSS06AC,IK18LAA,L07JA,MHM15JCTA}. 
	
	Let $e \ge 1$ be a given integer. Based on the  fact that  power GCD matrices $(S^e)$ on any finite  set $S$ of  positive integers
	are nonsinglular, Bourque and Ligh \cite{BL92LAA} first explored in 1992   the divisibility of power LCM matrices $[S^e]$ by power GCD matrices $(S^e)$ on $S$ in
	the ring $M_n(\mathbb{Z})$ of $n\times n$ matrices over the integers. Indeed, Bourque and Ligh \cite{BL92LAA} showed that  if
	$S=\{x_1, ..., x_n\}$ is factor closed (i.e. $ d\in S$ if $d \mid x$ for any $x \in S$), then the power GCD matrix $(S^e)$  always divides the power LCM matrix $[S^e]$ in	the ring $M_n(\mathbb{Z})$.
	That is, there is an integer matrix $M\in M_n(\mathbb{Z})$ such that $[S^e]=(S^e)M$ or
	$[S^e]=M(S^e)$(i.e. $(S^e)^{-1}[S^e]\in M_n(\mathbb{Z})$, or
	$[S^e](S^e)^{-1}\in M_n(\mathbb{Z})$.
	
	In 2002, Hong \cite{H02LAA} conducted an investigation into the divisibility of LCM matrices by GCD matrices on  gcd-closed set.
	The set $S$ is called {\it gcd closed} if $(x_i,x_j)\in S$ for all $1\le i,j
	\le n$. Clearly, any factor closed set is gcd closed but the converse is not true.
	Hong \cite{H02LAA} proved that for any gcd-closed set  $S$  with  $|S| \leq 3$, the GCD matrices $(S)$ divides the LCM matrices $[S]$. But for any integer  $n \geq 4$, there exists a gcd-closed set  $S$  with  $|S|=n$  such that  $(S)$ does not divide $[S]$. Consequently, Hong has posed the following open problem.
	
	\begin{problem}\cite{H02LAA}\label{P1}
		Let $n\geq 4$. Find the necessary and sufficient conditions on the gcd-closed set $S$ with $|S|=n$ such that the GCD matrix $(S)$ divides the LCM matrix $\left[S\right]$ in the ring $M_n(\mathbb{Z}).$
	\end{problem}
	
	A crucial perspective for investigating Problem \ref{P1} is to examine the structure of the divisibility relations among the elements of the gcd-closed set $S$. In 2003, Hong \cite{H03CM} proved that $(S)$ divides $[S]$ when $S$ is a divisor chain. For more general gcd-closed sets, one may need the concept of greatest-type divisor, which was defined by Hong in \cite{H99JA} to solve the Bourque-Ligh conjecture \cite{BL92LAA}. For $x,y\in S$ and $x<y$, we say that $x$ is a {\it greatest-type divisor}   of $y$ in $S$ if $x|y$ and the conditions $x|z|y$ and $z\in S$ 
	imply that $z\in \{x,y \}$. For $x\in S$, we
	denote by $G_S(x)$  the set of all greatest-type divisors of $x$ in $S$.
	In 2008, Hong, Zhao and Yin \cite{HZY08AA} constructed an integer matrix \( M \) such that for any gcd-closed set \( S \) with \( \max_{x \in S} \{|G_S(x)|\} = 1 \), the equation \( (S^e)M = [S^e] \) holds. This addressed Problem \ref{P1} for the case where $\max_{x \in S} \{|G_S(x)|\}=1$. To investigate the case where $\max_{x \in S} \{|G_S(x)|\} \ge 2$, Feng, Hong and Zhao\cite{FHZ09DM} introduced the following  key condition that was subsequently extended by Zhao, Chen and Hong in \cite{ZCH22JCTA}. 
	
	\begin{definition}\cite{ZCH22JCTA}
		We say that an element $x\in S$ with $ |G_S(x)| \ge 2 $ satisfies the 
		condition $ \mathcal{C} $ if $ [y,z] =x \ \text{and} \ (y,z) \in G_S(y) \cap G_S(z) $ for any $ y,z \in G_S(x) $ with $y\neq z$.  We further state that  the set $ S $ satisfies the condition $ \mathcal{C} $ if 
		each element $ x \in S $ with $ |G_{S}(x)|\geq 2 $ satisfies the condition $ \mathcal{C} $.
	\end{definition}
	
	In fact, Feng, Hong and Zhao \cite{FHZ09DM}  show that if $S$ is  gcd-closed  and  $\max_{x\in S}\{|G_S(x)|\}\leq 2,$ then $(S^e)|\left[S^e\right]$ if and only if either $\max_{x\in S}\{|G_S(x)|\}=1$ or $\max_{x\in S}\{|G_S(x)|\}=2$ with $S$ satisfying the condition $\mathcal{C}$. This result addresses Problem  \ref{P1} under the constraint that $ \max_{x\in S}\{|G_S(x)|\} \leq 2. $
	In 2014, Zhao \cite{Z14LMA} showed that $(S^e)$ does not divide $[S^e]$ when $S$ is gcd-closed with $\max_{x\in S}\{|G_S(x)|\} \ge 3$ and $|S| \le 7$. Furthermore, he proposed the following conjecture.
	\begin{conjecture} \cite{Z14LMA}\label{C1}
		Let $e\ge 1$ be an integer and  $S=\{x_1,\ldots,x_n\}$ be a gcd-closed set with $ \max_{x \in S} \{|G_{S}(x)|\}=m \geq 4 $. If $n<{m\choose 2} +m+2$, then $(S^e)$ does not divide $[S^e]$.
	\end{conjecture}
	
	In 2017, Altinsik, Yildiz and Keskin \cite{AYK17LAA} presented some certain gcd-closed sets on which $[S]$ does not divide by $(S)$, which implies that  Conjecture \ref{C1} might be true.  For $x\in S
	$ with  $|G_{S}(x)| \geq 2$, we say that $x$ satisfies the condition $ \mathcal{M} $ if for all distinct elements $ y,z \in G_S(x) $, it holds that $ [y,z] =x $.  Also,  we say that the set $ S $ satisfies the condition $ \mathcal{M} $
	if each element $ x \in S $ with $ |G_{S}(x)| \geq 2 $ satisfies the condition $ \mathcal{M} $.
	Furthermore, Altinsik, Yildiz and Keskin \cite{AYK17LAA} characterized the gcd-closed set \( S \) such that \( (S)|\left[S\right] \) for the case when \( |S|=8 \). This indicates that Problem \ref{P1} has been resolved for the case when \( |S|=8 \). At the end of 
	\cite{AYK17LAA},  they proposed the following conjecture as a generalization of Conjecture \ref{C1}.

	\begin{conjecture}\cite{AYK17LAA}\label{C2}
		Let $S$ be a gcd-closed set with $ \max_{x \in S} \{|G_{S}(x)|\}\geq 2 $. If $S$ does not satisfy the condition $\mathcal{M}$, then $(S)$ does not divide $[S]$.
	\end{conjecture}
	
	In 2022, Zhao, Chen and Hong \cite{ZCH22JCTA} proved that if \( S \) is gcd-closed and \( \max_{x\in S}\{|G_S(x)|\} = 3 \), then \( (S^e) | [S^e] \) if and only if \( S \) satisfies the condition \( \mathcal{C} \). Consequently, Problem \ref{P1} remains unresolved for the case where \( \max_{x\in S}\{|G_S(x)|\} \geq 4 \) and \( |S| \geq 9 \). In addition, they proposed the following conjecture.
	
	\begin{conjecture}\cite{ZCH22JCTA} \label{C3}
		Let $S$ be a gcd-closed set with $ \max_{x \in S} \{|G_{S}(x)|\} \geq 4 $ and  $e$ be a positive integer. Then the power LCM matrix $\left[S^e\right]$ is divided by the power GCD matrix $\left(S^e\right)$ in the ring $M_n\left(\mathbb{Z}\right)$ if and only if the set $S$ satisfies the condition $\mathcal{C}$.
	\end{conjecture}
	
	It is evident that if Conjecture \ref{C3} holds true, then both Conjecture \ref{C1} and Conjecture \ref{C2} must also be valid. It is important to note that when employing the methods outlined in \cite{ZCH22JCTA} to investigate Conjecture \ref{C3}, one may encounter two major obstacles.
	The first obstacle arises in  confirming  the necessity of  Conjecture \ref{C3}.  For any two elements \( x_k \) and \( x_m \) from the set \( S \), it proves to be difficult to ascertain the positivity of the algebraic sum of all least common multiples of \( x_k \) and \( y \), where \( y \) varies over the intersection of  \( S \) and the set of all divisors of \( x_m \).  	
	The second obstacle involves characterizing \(\gcd(Y)\) along with its greatest type divisors for any given subset \(Y\) of \(G_S(x)\), where \(x \in S\) satisfying condition \(\mathcal{C}\) and with \(|G_S(x)| \geq 4\). This characterization is crucial for establishing the sufficiency of Conjecture \ref{C3}.
	
	In this paper, we focus on Conjecture \ref{C3}. By applying insights from some results by Bourque and Ligh \cite{BL93LMA}, we prove that for any given positive integer \( x \) and an arbitrary set \( Y \) of proper divisors of \( x \), the alternating sum of \( \gcd(Z) \) is consistently greater than zero, where \( Z \) ranges over the subsets of \( Y \cup \{x\} \). This overcomes the first obstacle. Furthermore, by utilizing subset inclusion to explore the divisibility structure of gcd-closed sets through a lattice-theoretic perspective, we tackle the second obstacle. Consequently, we prove that Conjecture \ref{C3} is true, thereby completely resolving Problem \ref{P1}.
	
	This paper is organized as follows. In Section 2, we introduce some notations and preliminary lemmas that are necessary for the proof of our main result. In Section 3, we present several novel lemmas concerning the divisibility structure of elements  of  gcd-closed set and provide the proof of Conjecture \ref{C3}. The final section is dedicated to remarks that illustrate the relationship between the condition $\mathcal{C}$ and Boolean algebra, as well as an open problem.

	\section{Preliminaries}
	
	In this section, we introduce several notations and lemmas needed in the proof of our main result. 
	
	\subsection{Notations}
	
	For any given integer $t\ge 2$ and any $t$ positive integers $x_1,\ldots,x_t$, we denote the greatest common divisor of \( x_1, \ldots, x_t \) by \( (x_1, \ldots, x_t) \), and the least common multiple of \( x_1, \ldots, x_{t} \) by \( [x_1, \ldots, x_t] \).	
	In this paper, we will utilize the inclusion of subsets from the set of greatest-type divisors to investigate the divisibility structure of gcd-closed set $S$. We need to generalize the definition of the greatest common divisor and the least common multiple over positive integer sets. 
	
	Given any set $A$ of positive integers, we denote by $\gcd(A)$ and $\operatorname{lcm}(A)$  the greatest common divisor and the least common multiple of all the elements in $A$, respectively. For any  given sets $A$ and $B$ of  positive integers, One can easily verify the following properties. 
	
	\begin{itemize}
		\item[(i).] $\gcd(A\cup B)=\left(\gcd(A),\gcd(B)\right)$. In particular, it holds that $\gcd(A\cup B)=\gcd(A)$ if $B\subset A$;
		\item[(ii).]  If $B\subset C\subset A$ and $\gcd(A)=\gcd(B)$, then $\gcd(A)=\gcd(C)=\gcd(B)$;
		\item[(iii).] $[\gcd(A), \gcd( B)]=\gcd(B)$ if $B\subset A$;
		\item[(iv).]  $[x,(y,z)]=\big([x,y],[x,z]\big)$ holds for any positive integers $x,y$ and $z$.   Furthermore, $$[x,\gcd(A)]=\gcd\bigg(\bigcup\limits_{y\in A}\{[x,y]\}\bigg).$$
	\end{itemize}
	It is evident that properties (i)-(iv) also apply to the multiset \( A \), which includes repeated elements.	
	For any given finite index set $I=\{i_1,\ldots,i_t\}\subset \mathbb{{Z}^+}$, we denote a finite sequence $Y_I$  by
	$$Y_I:=\{y_{i_j}\}_{j=1}^t.$$ 
	If \( y_{i_1}, \ldots, y_{i_t} \) are distinct positive integers, then \( Y_I \) can be interpreted as a set. If there are repeated elements among $y_{i_1},\ldots,y_{i_t}$, then  $Y_I$ can be regarded as a multiset. Consequently, the definition of $\gcd(\cdot)$ can be extended to the finite sequence $Y_I$.	
	
	For any nonempty set \( A \) of positive integers, the proof of our main result necessitates the identification of the minimal gcd-closed set that contains \( A \). Let 
	$$\left<A\right>:=\bigcup_{B\subset A,B\neq \emptyset}\{\gcd(B )\}.$$
	In other words, if $A=\{z_1,\ldots,z_t\}$, then by the definition we have
	$$\left<A\right>=\{(z_{i_1},\ldots, z_{i_k})~|~1\leq k \leq t,~  1\leq i_1< \ldots< i_k\leq t\}.$$
	Clearly, the set $\left<A\right>$	is  the minimal gcd-closed set that contains \( A \).
	
	Throughout this section, we always assume that $S=\{x_1,\ldots,x_n\}$ is a gcd-closed set and $e$ is a given positive integer. 
	The following notations are crucial for characterizing the divisibility of power LCM matrices by power GCD matrices on gcd-closed set. For any given integer $x\in S$, we define 
	\begin{equation}\label{2}
		\alpha _{e,S}\left(x\right)
		:=\sum\limits_{d|x\atop d\nmid y,~y\in S,~ y<x} (\xi _e*\mu )(d).
	\end{equation}
	Furthermore, for all given integers $k$ and $m$ with $1\leq k,m\leq n$, we denote the function \( g(k,m) \) as follows:
	\begin{equation}\label{3}
		g\left( k,m \right) :=\frac{1}{\alpha _{e,S}\left( x_m \right)}\sum_{x_r|x_m}{c_{rm}\left[ x_k,x_r \right] ^e},
	\end{equation}
	where 
	\begin{align}\label{1}c_{ij}:=\sum_{\begin{subarray}{c}dx_i|x_j \\dx_i \nmid x_t, ~x_t<x_j
		\end{subarray}}\mu (d).
	\end{align}
	
	The following three sets, introduced by Zhao, Chen, and Hong \cite{ZCH22JCTA}, enhance our understanding of the divisibility structure inherent in gcd-closed sets that satisfy condition $\mathcal{C}$. For any $a,b\in S$, with $a|b$ and $a<b$, we denote three sets $A_S\left(a,b\right)$, $B_S\left(a,b\right)$ and $C_S\left(a,b\right)$ as follows:
	$$A_S\left(a,b\right):=\left\{u\in S:a|u|b,u\ne a\right\},$$
	$$B_S\left(a,b\right):=A_S\left(a,b\right)\cap G_S\left(b\right)$$
	and
	$$C_S\left(a,b\right):=A_S\left(a,b\right)\cup\left\{a\right\}.$$
	\subsection{Preliminary lemmas }
	
	First of all, we revisit the inverse of a GCD matrix on a gcd-closed set, as established by Bourque and Ligh \cite{BL93LMA}.

	\begin{lemma}\label{lem1}  \cite{BL93LMA}
		Let $S=\{x_1,\ldots,x_n\}$ be a gcd-closed set.
		Then the inverse of the power GCD matrix
		$(S^e)$ on $S$ is the matrix $W=(w_{ij})$, where
		$$
		w_{ij}=\sum_{\begin{subarray}{c} x_i\mid x_k \\ x_j\mid x_k
		\end{subarray}} {c_{ik}c_{jk} \over \alpha_{e,S}(x_k)}.
		$$
	\end{lemma}	
	
	With the assistance of Lemma 1 and Example 1 presented in \cite{BL93LMA}, we derive the following result.
	
	\begin{lemma}\label{lem3}
		If $S=\{x_1,\ldots,x_n\}$ is a gcd-closed set, then $\alpha _{e,S}\left( x_i \right)>0$ for any  $1\leq i\leq n$.
	\end{lemma}
	\noindent{\bf Proof.}
	Since $S$ is a gcd-closed set, by Lemma 1 of $\cite{BL93LMA}$, we know that	$$	(S^e)=E\varDelta E^T,$$ where $\varDelta=diag(\alpha_{e, S}(x_1),\alpha_{e, S}(x_2),\ldots,\alpha_{e, S}(x_n))$ and $E=(e_{ij})$ are $n\times n$ matrices with $e_{ij}$ defined as follows:
	$$e_{ij}=\begin{cases}
		1& \text{ if } x_j|x_i\\
		0& \text{ otherwise }\\
	\end{cases}.$$
	Evidently, the matrix \( E \) is invertible. Consequently, the matrix \( (S^e) \) is congruent to \( \varDelta \). Referring to Example 1 in \cite{BL93LMA}, we know that \( (S^e) \) is positive definite. This implies that \( \varDelta \) is also positive definite. Then  $\alpha _{e,S}\left( x_i \right)>0$ for any $1\le i\le n$.
	\hfill\ensuremath{\mbox{$\Box$}}
	\\

	The following lemma, established by Hong \cite{H04JA}, provides an efficient method for evaluating the denominator $\alpha_{e,S}(x_m)$ of $g(k,m)$.
	
	\begin{lemma}\cite{H04JA}\label{lem2}
		Let $S$ be a gcd-closed set, $L=\{1,\ldots,l\}$ be an index set,  $x\in S$ and 
		$G_S\left(x\right)=Y_L$
		be the set of the
		greatest-type divisors of $x$ 	in $S$. Then
		$$	\alpha _{e,S}\left(x\right)
		=x^e+\sum_{I\subset L,I\neq \emptyset}(-1)^{|I|}
		\big(x,\gcd(Y_I)\big)^e.$$
	\end{lemma}
	
	Altinişik, Yildiz and Keskin \cite{AYK17LAA} presented a formula for calculating \( c_{rm} \) in Lemma \ref{lem5}.
	
	\begin{lemma}\label{lem5}  \cite{AYK17LAA}
		Let $S=\{x_1,\ldots,x_n\}$ be a gcd-closed set  and $1\le r, m\le n$. If   $G_S\left(x_m\right)=Y_L$ with $L=\{1,\ldots,l\}$, then
		$$c_{rm}=\sum_{dx_r|x_m}\mu\left( d \right) +\sum_{I\subset L,I\neq \emptyset}(-1)^{|I|}\sum_{dx_r|\gcd(Y_I)}\mu\left( d \right).$$
	\end{lemma}
	
	Consequently, we get a formula for the numerator of $g(k,m)$ as follows.

	\begin{lemma}\label{lem6}
		Let $S=\{x_1,\ldots,x_n\}$ be a gcd-closed set and $1\le k, m\le n$.  If   $G_S\left(x_m\right)=Y_L$ with $L=\{1,\ldots,l\}$,  then
		\begin{equation*}
			\sum_{x_r|x_m}{c_{rm}\left[ x_k,x_r \right] ^e}=\left[ x_k,x_m \right] ^e+\sum_{I\subset L,I\neq \emptyset}(-1)^{|I|}\big[ x_k,\gcd(Y_I)\big] ^e.
		\end{equation*}	
		
	\end{lemma}
	\begin{proof}
		We begin by observing that for any integers \( x \) and \( y \) with \( x \) divides \( y \),   
		\begin{align}\label{e2.3}
			\sum_{dx|y}\mu(d)=\sum_{d|\frac{y}{x}}\mu(d)=\left\{
			\begin{array}{cc}
				1 &~~ x=y\\
				0 &~~\text{ otherwise }\\
			\end{array}
			\right..
		\end{align}
		It 	then follows from Lemma \ref{lem5} and (\ref{e2.3}) that 					
		\begin{align}
			\sum_{x_r|x_m}{c_{rm}\left[ x_k,x_r \right] ^e}&=\sum_{x_r|x_m}{\bigg(\sum_{dx_r|x_m}\mu\left( d \right) +\sum_{I\subset L,I\neq \emptyset}(-1)^{|I|}\sum_{dx_r|\gcd(Y_I)}\mu\left( d \right)\bigg) \left[ x_k,x_r \right] ^e}\nonumber\\
			&=\left[ x_k,x_m \right] ^e+\sum_{I\subset L,I\neq \emptyset}(-1)^{|I|}\sum_{x_r|x_m}\left[ x_k,x_r \right] ^e\sum_{dx_r|\gcd(Y_I)}\mu( d )\nonumber\\
			&=\left[ x_k,x_m \right] ^e+\sum_{I\subset L,I\neq \emptyset}(-1)^{|I|}\big[ x_k,\gcd(Y_I)\big] ^e,\nonumber
		\end{align}	
		which means that  Lemma \ref{lem6} is true.   \hfill\ensuremath{\mbox{$\Box$}}
	\end{proof}
	
	We now present a key lemma established by Zhao, Chen and Hong \cite{ZCH22JCTA}, which characterizes the properties of gcd-closed sets that satisfy condition $\mathcal{C}$.
	
	\begin{lemma}\label{lem7} \cite{ZCH22JCTA}
		Let $S$ be a gcd-closed set, $x\in S$ satisfy $|G_S\left(x\right)|\geq 2$, and $y\in G_S\left(x\right)$. Let $z\mid x$ and $z\nmid y$. If the set  $A_S\left(z,x\right)$ satisfies the condition  $\mathcal{C}$, then  $|G_S\left(u\right)|\geq 2$, and $\left[y,u\right]=\left[y,z\right]=x$ for any $u\in A_S\left(z,x\right).$
	\end{lemma}
	
	The following two lemmas provide essential information regarding $(S^e) \mid [S^e]$ when $S$ is a gcd-closed set with $\max_{x \in S} \{|G_{S}(x)|\} \leq 3$.
	
	\begin{lemma}\label{lem2.8}\cite{ZCH22JCTA}
		Let $S=\left\{x_1,\ldots,x_n\right\}$ be a gcd-closed set, 	
		$1\leq k,m \leq n$ and $1\le |G_S(x_m)|\le 3$. If the set $C_S((x_k,x_m),x_m)$ satisfies the condition $\mathcal{C}$, then  $g\left(k,m\right)\in \mathbb{Z}$.
	\end{lemma}
	
	\begin{lemma}\label{lem2.9}\cite{ZCH22JCTA}
		Let $S=\{x_1,\ldots,x_n\}$ be a gcd-closed set,  $x_m\in  S$ with $2\le |G_S(x_m)|\le 3$. If $x_m$ does not satisfies the condition $\mathcal{C}$,  then there exists an integer $1\le k\le n $ such that $g\left(k, m\right)\notin \mathbb{Z}$. 
	\end{lemma}
	
	Finally, we introduce a principle of duality concerning  the greatest common divisor and the least common multiple. 
	\begin{lemma}\label{lem4} \cite{H82S}
		For any given positive  integers $a_1,\ldots,a_{k-1}$ and $a_k$, 
		$$(a_1,a_2,\ldots,a_k)=\prod_{s=1}^k\prod_{1\le t_1<\ldots<t_s\le k}[a_{t_1},\ldots,a_{t_s}]^{(-1)^{s-1}}.$$
	\end{lemma}

	Note that, Lemma \ref{lem4} can be reformulated by our notation as follows. For any given nonempty set $A$ of positive integers, 
	$$\gcd(A)=\prod_{B\subset A,B\neq \emptyset }\operatorname{lcm}(B)^{(-1)^{|B|-1}}.$$

	\section{Main results}
	In this section, we first state our main result as follows.
	\begin{theorem}\label{thm1}
		Let $e$ be a positive integer, and $S$ be a gcd-closed set with $ \max_{x \in S} \{|G_{S}(x)|\} \geq 4 $. Then the power LCM matrix $\left[S^e\right]$ is divisible by the power GCD matrix $\left(S^e\right)$ in the ring $M_n\left(\mathbb{Z}\right)$ if and only if the set $S$ satisfies the condition $\mathcal{C}$.
	\end{theorem}
	
	In the following, we divide the proof of Theorem \ref{thm1} into  sufficiency part and necessity part. 
	
	\subsection{Proof of the sufficiency of  Theorem \ref{thm1}}
	
	In this subsection, we will prove the sufficiency of Theorem \ref{thm1}. Firstly, we need to prove some crucial properties of the gcd-closed set $S$ satisfying  the condition $\mathcal{C}$.

	\begin{lemma}\label{lem3.1}
		Let $S$ be a gcd-closed set, $x\in S$ satisfy  the condition $\mathcal{C}$ and   $G_S\left(x\right)=Y_L$ where $L=\{1,\ldots,l\}$  with $l\ge 3$. If   $I$ and $J$  are two nonempty distinct index subsets of  $ L$, then $$\gcd(Y_I)\neq \gcd(Y_J).$$
	\end{lemma}
	\begin{proof} We first claim that if $J\subset I$ satisfies that  $|I|-1=|J|$, then Lemma \ref{lem3.1} is true. In what follows we will prove the claim. 
		
		Suppose that $\gcd(Y_I)=\gcd(Y_J).$	
		Since $J\subset I\subset L$ and $|I|-1=|J|$, there exists unique integer $i_0\in I$ such that  $I\setminus J=\{i_0\}$.  Noticing that $x$ satisfies the condition $\mathcal{C}$,  we have $x=\left[y_i,y_j\right]$ for any integers $1\le i<j\le l$.  This implies that $\operatorname{lcm}(Y_K)=x$ for any index set  $K\subset L$ with $|K|\ge 2$.  It then follows from Lemma \ref{lem4} that 
		\begin{align}
			\gcd(Y_{J})&=\prod_{K\subset J,~K\neq \emptyset}\operatorname{lcm}(Y_K)^{(-1)^{|K|-1}}		
			\nonumber\\
			&=\bigg(\prod_{j\in J}y_j\bigg)\bigg(	\prod_{K\subset J,~|K|\ge2}x^{(-1)^{|K|-1}}\bigg)\nonumber\\	
			&=\bigg(\prod_{j\in J}y_j\bigg)\bigg(x^{-\sum\limits_{s=2}^{|J|}(-1)^{s}{|J|\choose s}}\bigg)\nonumber\\
			&=\bigg(\prod_{j\in J}y_j\bigg) x^{1-|J|}. \label{3.1}
		\end{align}
		Similarly, we also have  
		\begin{align}
			\gcd(Y_{I})=\bigg(\prod_{i\in I}y_i\bigg) x^{1-|I|}. \label{3.2}
		\end{align} 
		Noticing that  $I\setminus J=\{i_0\}$, 
		we deduce  by  (\ref{3.1}) and (\ref{3.2}) that 
		\begin{align}
			\gcd(Y_{I})=y_{i_0}\bigg(\prod_{j\in J}y_j\bigg)
			x^{1-|I|}=y_{i_0}x^{-1}\gcd(Y_{J}). \label{e3.3}
		\end{align}
	By our supposition  $\gcd(Y_I)=\gcd(Y_J)$ and (\ref{e3.3}), we  obtain  that
		$x= y_{i_0}$.  This leads to a contradiction with $y_{i_0}\in G_S(x)$. So the Claim  is proved.
		
		Now we return to the proof of Lemma \ref{lem3.1}. Assume that 
		$\gcd(Y_I)=\gcd(Y_J).$  Since $I\neq J$,  there must exist some integers either $i\in I\setminus J$ or   $j\in J\setminus I$. Without loss of generality, we let $i_1  \in I\setminus J$.  Let $H$ be the union of sets $I$ and $J$. Denote $H=\{h_1,\ldots,h_a, i_1\}$  and  $T=\{x,y_{h_1},...,y_{h_a},y_{i_1}\}$. Since  $G_S\left(x\right)=Y_L$, by the definition of $\left< T\right>$, we deduce that  the set $\left< T\right>$ is a gcd-closed set and $$G_{\left< T\right>}\left(x\right)=\left\{y_{h_1},...,y_{h_a},y_{i_1}\right\}=Y_H.$$ It then follows from the above Claim that 
		\begin{align}\label{3.3}
			\gcd(Y_H)\neq \gcd(Y_{H\setminus\{i_1\}}).
		\end{align}
		But by  our assumption   $\gcd(Y_I)=\gcd(Y_J)$  and  the definition of $H$, we get
		\begin{align*}
			\gcd(Y_H)=\gcd(Y_{I\cup J})=(\gcd(Y_I),\gcd(Y_J))=\gcd(Y_J).
		\end{align*}
		Then by the fact that $J\subset H\setminus\{i_1\} \subset H $, we deduce that 
		\begin{align*}
			\gcd(Y_J)
			=\gcd(Y_{H\setminus\{i_1\}})=\gcd(Y_H),
		\end{align*}
		which  contradicts to (\ref{3.3}). This means that $\gcd(Y_I)\neq \gcd(Y_J).$
		
		So Lemma \ref{lem3.1} is proved. 
		\hfill\ensuremath{\mbox{$\Box$}}
	\end{proof}
	
	\begin{lemma}\label{lem3.2}
		Let $S$ be a gcd-closed set,  $L=\{1,\ldots,l\}$ be an index set  with $l\ge 3$,  $x\in S$ with $G_S\left(x\right)=Y_L$ and $A_S(\gcd(Y_L),x)$ satisfy  the condition $\mathcal{C}$. Then for any given integer  $i\in L$ and any nonempty index set $I\subset L\setminus \{i\}$, 	 the following two statements are true:
		
		(i).  $$\gcd\big(Y_{I\cup \{i\}}\big)\in G_S\left(\gcd(Y_I)\right).$$
		
		(ii).  If  $z\in S$ and $(z,x)\nmid y_i,$  then $$\big(z, \gcd(Y_I)\big)\nmid y_i.$$ 
	\end{lemma}
	\begin{proof}
		Let $i\in L$ be any given integer and $I$ be any nonempty  subset of  $L\setminus \{i\}$  with $|I|=t.$ Then  $1\le t<l$. In what follows we  use induction on $t$ to prove Lemma \ref{lem3.2}.
		
		We first let $t=1$. Then we may let $I=\{j\}$, that is $Y_I=\{y_j\}=\{\gcd(Y_I)\}$. Since $x\in A_S(\gcd(Y_L),x)$ and $A_S(\gcd(Y_L),x)$ satisfies  the condition $\mathcal{C}$,  we know that 
		$x$ satisfies the condition $\mathcal{C}$.  Then by the fact that   $\{y_i,y_{j}\}\subset Y_L$ and $G_S\left(x\right)=Y_L$, we obtain that 
		$$\gcd\big(Y_{I\cup \{i\}}\big)=\left(y_{i},y_{j}\right)\in G_S\left(y_{j}\right)=G_S\left(\gcd(Y_I)\right)$$ as desired. So  Lemma \ref{lem3.2} (i) is proved  when $t=1$.

		Now we prove that $\left(z, y_j\right)\nmid y_i$. Otherwise, we suppose that $\left(z, y_j\right)| y_i$. Since  $\{y_i,y_{j}\}\subset G_S\left(x\right)$ and  $x$ satisfies condition $\mathcal{C}$, we have   $\left[y_i,y_j\right]=x$. So  
		$$\left(z,x\right)=\left(z,\left[y_i,y_j\right]\right)=\left[\left(z,y_i\right),\left(z,y_j\right)\right]. $$
		Thus by  our supposition  $\left(z,y_j\right)|y_i$, we get  $\left(z,x\right)|y_i$. This  contradicts to the fact $\left(z,x\right)\nmid y_i$. Hence $(z,y_j)=(z,\gcd(Y_I))\nmid y_i$, which implies that  Lemma \ref{lem3.2} (ii) holds for the case $t=1$.		
		The proof  of  Lemma \ref{lem3.2}  for the case $t=1$ is complete.

		Now we let $2\le t<l$.  Assume that  Lemma \ref{lem3.2}  is true for the case $t-1$.  Next, we consider the case $t$. Let $j_0$ be any given integer of $I$, and denote $J=I\setminus\{j_0\}$.  Since  $I\subset L\setminus \{i\}$, we have $|J|=|I|-1=t-1\ge 1$ and $J\subset L\setminus \{i,j_0\}$.
		Then by the induction hypothesis, 
		we deduce that 
		\begin{align}\label{3.4}
			\gcd\big(Y_{J\cup \{i\}}\big)\in G_S\left(\gcd(Y_J)\right),
		\end{align}
		\begin{align}\label{3.5}
			\gcd\big(Y_I\big)=\gcd\big(Y_{J\cup \{j_0\}}\big)\in G_S\left(\gcd(Y_J)\right)
		\end{align}
		and 
		\begin{align}\label{e3.6}
			\big(z, \gcd(Y_J)\big)\nmid y_i.
		\end{align}
		Noticing that  $J\neq L$, we obtain  by Lemma \ref{lem3.1} that $$\gcd(Y_J)\neq\gcd(Y_L),$$
		which implies that $\gcd(Y_J)\in  A_S(\gcd(Y_L),x)$. And so $\gcd(Y_J)$ satisfies the condition $\mathcal{C}$ since  $A_S(\gcd(Y_L),x)$ satisfies  the condition $\mathcal{C}$. It then follows from (\ref{3.4}) and (\ref{3.5}) that
		\begin{align}\label{e3.7}
			\big(\gcd\big(Y_{J\cup \{i\}}\big),\gcd(Y_I)\big)\in G_S\big(\gcd(Y_I)\big)
		\end{align}
		and 
		\begin{align}\label{e3.8}
			\big[\gcd\big(Y_{J\cup \{i\}}\big),\gcd(Y_I)\big]= \gcd(Y_J).
		\end{align}
				Noticing that
		\begin{align}\label{e3.9}
			\gcd\big(Y_{I\cup \{i\}}\big)&=\gcd\big(Y_{J\cup \{i\}\cup\{j_0\}}\big)\nonumber\\
			&=\big(\gcd\big(Y_{J\cup \{i\}}\big),\gcd\big(Y_{J\cup \{j_0\}}\big)\big)\nonumber\\
			&=\big(\gcd\big(Y_{J\cup \{i\}}\big),\gcd(Y_I)\big),
		\end{align}
		 we deduce from (\ref{e3.7}) and (\ref{e3.9}) that  Lemma \ref{lem3.2} (i) is true  for the case $t$.  
		
		It remains to  show that $(z,\gcd(Y_I))\nmid y_i$. Otherwise, we suppose that $(z,\gcd(Y_I))\mid y_i$.	
		From (\ref{e3.9}),  we obtain that 
		\begin{align}
			\left(z,\gcd(Y_J)\right)&=\big(z,\big[\gcd\big(Y_{J\cup \{i\}}\big),\gcd(Y_I)\big]\big)
			=\big[(z,\gcd\big(Y_{J\cup\{i\}}),(z,\gcd(Y_I))\big]. \label{3.6}
		\end{align}
		Then by our assumption that $(z,\gcd(Y_I))\mid y_i$ and the fact that  $\gcd\big(Y_{J\cup\{i\}})\mid y_i$, we can  deduce  by  (\ref{3.6}) that 
		$\left(z,\gcd(Y_J)\right)|y_i$.	
		But by (\ref{e3.6}) we know that $\big(z, \gcd(Y_J)\big)\nmid y_i.$  This is  a contradiction. 
		Hence $(z,\gcd(Y_I))\nmid y_i$ as desired. So Lemma \ref{lem3.2} (ii) is proved  for the case $t$.
		
		The proof  of  Lemma \ref{lem3.2}  is complete. 
		\hfill\ensuremath{\mbox{$\Box$}}
	\end{proof}
	
	\begin{lemma}\label{lem3.3}
		Let $S=\left\{x_1,\ldots,x_n\right\}$ be a gcd-closed set, 	
		$1\leq k,m \leq n$ and $|G_S(x_m)|\ge4$. If the set $C_S((x_k,x_m),x_m)$ satisfies the condition $\mathcal{C}$, then  $g\left(k,m\right)\in \mathbb{Z}$.
	\end{lemma}
	
	\begin{proof} 	
		Let $G_S(x_m)=Y_L$ with  $L=\{1,\ldots,l\}$ and $l\ge 4$.  In order to prove Lemma \ref{3.3}, we only need to consider two cases  depending on whether $(x_k,x_m)$ divides $\gcd(Y_L)$ or not.
		
		{\sc Case 1.} $(x_k,x_m)|\gcd(Y_L)$.  For any given index set $I\subset L$,  we have $$\gcd(Y_L)\mid \gcd(Y_I)\mid x_m.$$ And so
		\begin{align}\label{e3.12}
			(x_k,\gcd(Y_L))\mid (x_k,\gcd(Y_I))\mid(x_k,x_m)	
		\end{align}
		Since $(x_k,x_m)|\gcd(Y_L)$ and $\gcd(Y_L)|x_m$, we get $$(x_k,x_m)=((x_k, x_m),\gcd(Y_L) )=(x_k,\gcd(Y_L)). $$ Then by (3.12)  we obtain that 
		$$(x_k,\gcd(Y_L))=(x_k,\gcd(Y_I))=(x_k,x_m).$$
		Therefore 
		$$\left[x_k,\gcd(Y_I)\right]^e=\frac{x_{k}^{e}\gcd(Y_I)^e}{( x_k,\gcd(Y_I)) ^e}=\frac{x_{k}^{e}\gcd(Y_I)^e}{( x_k,x_m ) ^e}.
		$$
		It then follows from Lemmas \ref{lem2} and  \ref{lem6}  that
		\begin{align*}
			g\left(k,m\right)&=\frac{1}{\alpha_{e,S}(x_m)}
			\bigg([ x_k,x_m] ^e+\sum_{I\subset L,I\neq \emptyset}(-1)^{|I|}\big[ x_k,\gcd(Y_I)\big]^e\bigg)\\
			&=\frac{1}{\alpha_{e,S}(x_m)}\bigg(\frac{x_{k}^{e}x_{m}^{e}}{( x_k,x_m )^e}+\sum_{I\subset L,I\neq \emptyset}(-1)^{|I|}\frac{x_{k}^{e}}{( x_k,x_m ) ^e}\gcd(Y_I)^e\bigg)\\
			&=\frac{x_{k}^{e}}{( x_k,x_m) ^e}\frac{1}{\alpha _{e,S}(x_m)}\bigg( x_{m}^{e}+\sum_{I\subset L,I\neq \emptyset}(-1)^{|I|} \gcd(Y_I)^e \bigg) \\
			&=\frac{x_{k}^{e}}{\left( x_k,x_m \right) ^e}\in \mathbb{Z} 
		\end{align*}
		as desired.
		
		{\sc Case 2.}  $\left(x_k,x_m\right)\nmid \gcd(Y_L)$. We first let $(x_k,x_m)=x_m$. Since $G_S(x_m)=Y_L$,  we obtain that for any index set $I\subset L$,
		$$\left[x_k,\gcd(Y_I)\right]^e=x_k^e.$$
	Then by Lemma \ref{lem6} we deduce that
		\begin{align*}
			g\left(k,m\right)&=\frac{1}{\alpha_{e,S}(x_m)}\bigg(x_k^e+\sum_{I\subset L,I\neq \emptyset}(-1)^{|I|}x_k^e\bigg)\\
			&=\frac{x_k^e}{\alpha_{e,S}(x_m)}\bigg(1+\sum_{I\subset L,I\neq \emptyset}(-1)^{|I|}\bigg) \\
			&=\frac{x_k^e}{\alpha_{e,S}(x_m)}\bigg(1+\sum_{s=1}^{|L|}( -1 ) ^s{|L|\choose s}\bigg)\\
			&=0\in \mathbb{Z} 
		\end{align*}
		as required. 
		
		Now  let $(x_k,x_m)<x_m$. Then we derive that the set $$B_S\left(\left(x_k,x_m\right),x_m\right)=A_S\left((x_k,x_m),x_m\right)\cap G_S\left(x_m\right)$$ is nonempty. Since $\left(x_k,x_m\right)\nmid \gcd(Y_L)$, we have  $B_S\left(\left(x_k,x_m\right),x_m\right)\neq G_S\left(x_m\right)$. Then there is an  integer $ j\in L$ such that  $y_j\notin B_S\left(\left(x_k,x_m\right),x_m\right).$
		Without loss of generality, let $y_l\notin B_S\left(\left(x_k,x_m\right),x_m\right).$
		So by the definition of $B_S\left(\left(x_k,x_m\right),x_m\right)$, we deduce that $(x_k,x_m)\nmid y_l.$ In what follows we will prove that $g(k,m)=0$.
		
		We first show that $\left[x_k,x_m\right]=\left[x_k,y_l\right]. $
		Since $B_S\left(\left(x_k,x_m\right),x_m\right)$ is nonempty, we can select an element $y_i$ from  the set  $B_S\left(\left(x_k,x_m\right),x_m\right)$. Then $(x_k,x_m)|y_i$.  Since $\left(x_k,x_m\right)\nmid y_l$ and the set $C_S((x_k,x_m),x_m)$ satisfies the condition $\mathcal{C}$,  we obtain by Lemma $\ref{lem7}$ that
		$\left[\left(x_k,x_m\right),y_l\right]=x_m.$
		And so
		\begin{align}\label{3.9}
			\left[x_k,x_m\right]=\left[x_k,\left[\left(x_k,x_m\right),y_l\right]\right]=\left[x_k,y_l\right]
		\end{align}
		as desired.
		
		Secondly, we will show that 
		$$\left[ x_k,\gcd(Y_I) \right]=\left[ x_k,\gcd(Y_{I\cup\{l\} })\right] $$
		for any given index set $I\subset L\setminus \{l\}$. 
		On the one hand, since  $C_S((x_k,x_m),x_m)$ satisfies the condition $\mathcal{C}$ and  $\left(x_k,x_m\right)\nmid y_l$, by Lemma \ref{lem3.2} (ii), we have
		$\left(x_k,\gcd(Y_I)\right)\nmid y_l,$
		which implies that
		\begin{align}\label{3.10}
			\left(x_k,\gcd(Y_I)\right)\nmid  \gcd(Y_{I\cup\{l\} }).
		\end{align}
		
		On the other hand, since  $I\subset L\setminus \{l\}$ and $C_S((x_k,x_m),x_m)$ satisfies the condition $\mathcal{C}$,  by Lemma \ref{lem3.2} (i),  we obtain that 
		\begin{align}\label{3.11}
			\gcd(Y_{I\cup\{l\} })\in G_S(\gcd(Y_I)).
		\end{align}
		Taking $z=\left(x_k,\gcd(Y_I)\right)$, $y= \gcd(Y_{I\cup\{l\} })$ and $x=\gcd(Y_I)$ in Lemma \ref{lem7}, we deduce from  (\ref{3.10}) and  (\ref{3.11})that 
		\begin{align*}
			\left[\left(x_k,\gcd(Y_I)\right),	\gcd(Y_{I\cup\{l\} })\right]=\gcd(Y_I).
		\end{align*} 
		It then follows that
		\begin{align}\label{3.12}\left[ x_k,\gcd(Y_I) \right] =\left[x_k,[\left(x_k,\gcd(Y_I)\right),	\gcd(Y_{I\cup\{l\} })]\right]=\left[ x_k,\gcd(Y_{I\cup\{l\} }) \right].
		\end{align}
		
		Finally, by Lemma \ref{lem6}, (\ref{3.9}) and (\ref{3.12}), we can deduce that 
		\begin{align*}
			\sum_{x_r|x_m}{c_{rm}\left[ x_k,x_r \right]^e}=&[ x_k,x_m] ^e+\sum_{I\subset L,I\neq \emptyset}(-1)^{|I|}\big[ x_k,\gcd(Y_I)\big]^e\\
			=&\left[ x_k,x_m \right] ^e+\sum_{I\subset L\setminus\{l\},I\neq \emptyset}(-1)^{|I|}\big[ x_k,\gcd(Y_I)\big]^e\\
			&-\left[ x_k,y_l \right] ^e+\sum_{I\subset L\setminus\{l\},I\neq \emptyset}(-1)^{|I|+1}\left[x_k, \gcd(Y_{I\cup\{l\} }) \right]^e\\
			=&0.
		\end{align*}
		Therefore, by (\ref{3}), we can immediately conclude that $ g(k,m) = 0 \in \mathbb{Z} $ as desired.
		
	This completes the proof of Lemma \ref{lem3.3}. \hfill\ensuremath{\mbox{$\Box$}}
	\end{proof}\\
	
	We are now in a position to give the proof of the sufficiency of Theorem \ref{thm1}.\\
	
	\noindent{\bf Proof of the sufficiency.} Assume that  $S$ satisfies the condition $\mathcal{C}$. By Lemma $\ref{1}$ we derive that
	\begin{align}
		\left( \left[ S^e \right] \left( S^e \right) ^{-1} \right) _{ij}&=\sum_{t=1}^n{\left[ x_i,x_t \right] ^e}\sum_{\begin{subarray}{c}
				x_t|x_k\\
				x_j|x_k\\
		\end{subarray}}{\frac{c_{tk}c_{jk}}{\alpha _{e,S}\left( x_k \right)}}\nonumber\\
		&=\sum_{x_j|x_k}{c_{jk}}\sum_{x_t|x_k}{\frac{c_{tk}\left[ x_i,x_t \right] ^e}{\alpha _{e,S}\left( x_k \right)}}\nonumber\\
		&=\sum_{x_j|x_k}{c_{jk}g\left( i,k \right)},\label{3.17}
	\end{align} 
	for all integers $1\leq i,j\leq n$. By the definition of $c_{jk}$, we have  $c_{jk}\in \mathbb{Z}$. Since the gcd-closed set $S$ satisfies the condition $\mathcal{C}$, by Lemmas \ref{lem2.8} and \ref{lem3.3}, we get $g\left(i,j\right)\in\mathbb{Z}$  for all integers  $1\leq i,j\leq n$.  It then follows from (\ref{3.17}) that $$\left( \left[ S^e \right] \left( S^e \right) ^{-1} \right) _{ij}\in \mathbb{Z},~~~\forall~1\leq i,j\leq n, $$
	which implies that $(S^e)\mid [S^e]$.
	
	The sufficiency of  Theorem \ref{thm1}
	is proved.
	\hfill\ensuremath{\mbox{$\Box$}}

	\subsection{Proof of the necessity of Theorem \ref{thm1}}
	
	In this subsection, we present the proof of the necessity of Theorem \ref{thm1}. We begin with a key lemma about an alternating sum for determining the non-integrality of $g(s,m)$ for some integers $1\le s,m\le n$.
	\begin{lemma}\label{lem8}
		Let $x, e,k$ be  given positive integers and  $K=\{1,\ldots,k\}$ be an index set. 
		If $y_1,\ldots,y_k$ are proper  divisors of $x$, then
		\begin{align}\label{e3.13}
			{x}^e+\sum_{I\subset K,I\neq \emptyset}(-1)^{|I|}\gcd(Y_I)^e>0.
		\end{align}
		\end{lemma}

	\noindent\textbf{Proof.}  
	We first define the set  
	$$
	Y:=\bigcup_{I\subset K}\{(x, \gcd(Y_I))\}.
	$$
	Clearly, the set $Y$ is a gcd-closed set. By the definition of the greatest-type divisor, and noticing that  $y_1,\ldots,y_k$ are proper divisors of $x$, we can conclude that the set  $G_Y(x)$ is nonempty and  comprises certain elements from \( y_1, \ldots, y_k \). 
	Without loss of generality, we may let $G_K(x)=Y_L$ with $L=\{1,\ldots,l\}$. So $1\leq |G_Y\left(x\right)|=l\leq k$. 
	For convenience, for any given integer $j$ with $l\le j\le k$, we define  
	\begin{align*}
		\alpha_j:={x}^e+\sum_{I\subset J,I\neq \emptyset}(-1)^{|I|}\gcd(Y_I)^e, 
	\end{align*}
	where $J=\{1,...,j\}$.
	Since $Y$ is a gcd-closed set and  $G_Y(x)=Y_L$, we derive from Lemmas \ref{lem3} and  \ref{lem2}  that 
	\begin{align*}
		\alpha_l={x}^e+\sum_{I\subset L,I\neq \emptyset}(-1)^{|I|}\gcd(Y_I)^e=\alpha_{e,Y}(x)>0.
	\end{align*}
	So in order to prove Lemma \ref{lem8},
	we only need to  prove that $\alpha_k=\alpha_l$. Clearly if $l=k$, then $\alpha_k=\alpha_l$. 
	
	Now we deal with the remaining case $1\le l \le k-1$. 
	To prove $\alpha_k=\alpha_l$, it suffice to show that for each integer  $j$ with $l+1\le j\le k$, the equation $$\alpha_{j-1}=\alpha_{j}$$
	holds  which will be done in what follows.
	
	Let $j$ be any given integer with   $l+1\le j\le k$.  Then by $G_Y(x)=Y_L$ and the definition of $Y$, we deduce that  there  must exist  an integer $i$ such that $y_{i}\in Y_L$ and $y_{j}\mid y_{i}$. This means that for any index set $I\subset J\setminus \{i,j\}, $
	\begin{align}\label{e3.14}
		\gcd(Y_{I\cup\{j\}})=\gcd(Y_{I\cup\{i,j\}}).
	\end{align} 
	It then follows from the definition of 	$\alpha_{j}$ and (\ref{e3.14}) that 
	\begin{align*}
		\alpha_{j}&=\alpha_{j-1}-y_j^e+\sum_{I\subset J\setminus\{j\}}(-1)^{|I|+1}\gcd(Y_{I\cup\{j\}})^e \\
		&=\alpha_{j-1}+(y_i,y_j)^e-y_j^e+\sum_{I\subset J\setminus\{i,j\}}\bigg((-1)^{|I|+1}\gcd(Y_{I\cup\{j\}})^e+(-1)^{|I|+2}\gcd(Y_{I\cup\{i,j\}})^e\bigg)\\
		&=\alpha_{j-1}
	\end{align*}
	as desired. 
	
	The proof of  Lemma \ref{lem8} is complete.\hfill\ensuremath{\mbox{$\Box$}}\\
	
	\noindent{\bf Remark 1.} Note that the set $\{y_1,\ldots,y_k\}$ is allowed to be a multiset with 
	$y_i=y_j$ for some integers $1\le i\neq j\le k$.
	
	\begin{lemma}\label{lem9}
		Let $S=\{x_1,\ldots,x_n\}$ be a gcd-closed set,  $x_m\in  S$ with $|G_S(x_m)|\ge 4$. If $x_m$ does not satisfy the condition $\mathcal{C}$,  then there exists an element $x_k$ in $ S$ such that $g\left(k, m\right)\notin \mathbb{Z}$. 
	\end{lemma}
	
	\noindent\textbf{Proof.} Let $G_S(x_m)=Y_L$ with $L=\{1,\ldots,l\}$ and $l\ge 4$. Since $x_m$ does not satisfy the condition $\mathcal{C}$, we consider the following two cases.
	
	{\sc Case 1.} {\it There exists two integers $i$ and $j$ with $1\leq i<j\leq l$ such that $\left[y_i, y_j\right]<x_m$.} 
	
	Firstly, for any given integer $1\leq i\leq l$, we define  $$Y_{L}(i):=\left\{y_j:~\left[y_i,y_j\right]<x_m,~j\in L,~j\neq i
	\right\}.$$ 
	In Case 1,  we know that there must exist an integer $i\in L$, such that $Y_L(i)$ is nonempty. 
	Then we may let 
	$$a:=\min\{|Y_L(i)|:~Y_L(i)\neq \emptyset,~1\le i\le l\}.$$
	It is clear that  $1\le a<l$. Without loss of generality, we may let $Y_L(1)$ attain the minimal cardinality $a$, i.e.   $|Y_L(1)|=a$, and let
	\begin{align}\label{3.15}
		[y_1,y_j]<x_m,~\forall~2\le j\le a+1
	\end{align}
	and let
	\begin{align}\label{e3.15}
		[y_1,y_j]=x_m,~\forall~ a+2\le j\le l
	\end{align}
	if $a+2\le l$. Denote 
	$$Y:=\bigcup_{I\subset  A}\{\gcd(Y_{I\cup (L\setminus A)})\},$$
	where $A=\{1,\ldots,a\}$.
	Since 
	$S$ is a gcd-closed set,   we get $Y$ is a subset of $S$. 
	
	Now we claim that for any given element $x\in Y$, 
	\begin{align}\label{3.21}
		[y_i,x]<x_m,~\forall~ 1\le i\le l.
	\end{align}
	If $a+1\le i\le l$, then by the definition of $Y$, we get $x\mid y_i$. Using  the fact that $y_i$ is the  greatest-type divisor of $x_m$, we get
	$[y_i,x]=y_i<x_m$ as required.
	If $1\leq i\leq a$,  then by (\ref{3.15}) and the definition of $Y_L(i)$,  we deduce that $Y_L(i)\neq \emptyset$.  Hence  \begin{align}\label{e3.16}
		|Y_L(i)|\geq\min\limits_{j\in L, Y_L(j)\ne \emptyset}|Y_L(j)|=|Y_L(1)|= a.
	\end{align}	
	Note that  $y_i\not\in Y_L(i)$. Then we have $|Y_L(i)\cap Y_A|<a$. Thus by (\ref{e3.16}), we derive that $Y_L(i)\cap Y_{L\setminus A}\neq \emptyset$, i.e. 	
	there exists an integer $j_0\in L\setminus A$  such that $\left[y_i,y_{j_0}\right]<x_m$.	
It subsequently follows from the definition of \( Y \) and the fact that \( x \in Y \) that
	$$\left[y_i,x\right]\leq \left[y_i,y_{j_0}\right]<x_m$$
	as desired. The Claim is proved.
	
	Next, we  prove that
	\begin{align}\label{3.23}
		\gcd(Y_{L\setminus A})>\gcd(Y_L).
	\end{align}	 
	By (\ref{3.15}) and (\ref{e3.15}), we have
	\begin{align*}
		\bigcup_{j\in L\setminus A}\{[y_1,y_j]\}=
		\left\{\begin{array}{cc}
			\{[y_1,y_{a+1}]\}, &~~\text{if }~ a+1=l,\\
			\{[y_1,y_{a+1}],x_m\}, &~~\text{if }~ a+2\le l.
		\end{array}
		\right.
	\end{align*}
	Then we obtain that 
	\begin{align*}
		\left[y_1, \gcd(Y_{L\setminus A})\right]
		&=\gcd\bigg(\bigcup_{j\in L\setminus A}\{[y_1,y_j]\}\bigg)
		=\left[y_1,y_{a+1}\right]>y_1
	\end{align*}	
	since $\{y_1,y_{a+1}\}\in G_S(x_m)$. It then follows from the fact $\gcd(Y_L)\mid \big(y_1,\gcd(Y_{L\setminus A})\big)$ that 
	$$
	\frac{\gcd(Y_{L\setminus A})}{\gcd(Y_L)}\ge \frac{\gcd(Y_{L\setminus A})}{\big(y_1,\gcd(Y_{L\setminus A})\big)}=\frac{\left[y_1, \gcd(Y_{L\setminus A})\right]}{y_1}>1,
	$$
	which implies that (\ref{3.23}) is proved.

	Since $\gcd(Y_{L\setminus A})>\gcd(Y_L)$,  we  deduce that there  exists  an unique  integer $b$ with $1\le b\leq a$  such that for all index sets $I\subset A$ with $|I|=b$,
	\begin{align}\label{e3.23}
		\gcd(Y_{I\cup (L\setminus A) })=\gcd(Y_L),
	\end{align}
	and there is an index set  $I_0\subset A$   satisfying that $|I_0|=b-1$ and 
	\begin{align*}
		\gcd(Y_{I_0\cup(L\setminus A)  })>\gcd(Y_L).
	\end{align*}
	Without loss of generality, we may let 
	\begin{align}\label{e3.24}
		(y_{a-b+2},\ldots,y_a,y_{a+1},\ldots,y_l)> \gcd(Y_L). 
	\end{align}
	Denote $w=a-b+1$ and $W=\{1,\ldots,w\}$. Then (\ref{e3.24}) can be expressed concisely as 
	\begin{align}\label{3.33}
		\gcd(Y_{L\setminus W})> \gcd(Y_L). 
	\end{align}
	Since $\gcd(Y_{L\setminus W})\in Y\subset S$ and $S$ is a gcd-closed set, there is an unique integer $k$ such that  $$x_k=\gcd(Y_{L\setminus W}).$$
	In what follows we will prove that $g(k,m)\not\in\mathbb{Z}$, which implies that Lemma \ref{lem9} is true for the Case 1. 
	
	On the one hand, since  $x_k\in Y$ and $x_k\mid x_m$,  we obtain from the  Claim (\ref{3.21}) that $[y_i,x _k]<x_m$ for any integers $i\in L$, that is $[y_1,x_k],\ldots,[y_l,x_k]$ are proper divisors of  $x_m$.  Since $S$ is a gcd-closed set, $G_S(x_m)=Y_L$ and  $x_k\mid x_m$, we can derive from  Lemmas \ref{lem6},  \ref{lem8} and Remark 1  that 
	\begin{align}
		\sum_{x_r|x_m}{c_{rm}\left[ x_k,x_r \right] ^e}&=\left[ x_k,x_m \right] ^e+\sum_{I\subset L,I\neq \emptyset}{( -1 ) ^{|I|}\left[ x_k,\gcd(Y_I)\right]^e}\label{3.24}\\
		&=x_m^e+\sum_{I\subset L,I\neq \emptyset}( -1 ) ^{|I|}\gcd\bigg(
		\bigcup_{i\in I}\{[x_k,y_i]\}\bigg)^e\nonumber\\
		&>0.	\label{e3.17}
	\end{align}
	
	On the other hand,  we deduce from Lemma \ref{lem2} that 
	\begin{align}\nonumber
		\alpha _{e, S}\left( x_m \right)&=x_{m}^{e}+\sum_{I\subset L,I\neq \emptyset}(-1) ^{|I|}\gcd(Y_I)^e\\
		&=\beta(x_m)+\sum_{
			\begin{subarray}{c} 
				I\subset W,I\neq \emptyset\\ 
				J\subset L\setminus W
			\end{subarray}	
		}(-1) ^{|I|+|J|}\gcd(Y_{I\cup J})^e\nonumber\\
		&=\beta(x_m)+\sum_{i=1}^w\sum_{I\subset \{i+1,\ldots,l\}}(-1) ^{|I|+1}\gcd(Y_{I\cup \{i\}})^e,
		\label{e3.18}
	\end{align}
	where $$\beta(x_m):=x_{m}^{e}+\sum_{I\subset L\setminus W,I\neq \emptyset}(-1) ^{|I|}\gcd(Y_I)^e.$$
	Since $x_k=\gcd(Y_{L\setminus W})$, we have  $[x_k,\gcd(Y_I)]=\gcd(Y_I)$ for any index set $I\subset L\setminus W$.
	Then by (\ref{3.24}), we derive that  	\begin{align}\label{e3.19}
		\sum_{x_r|x_m}{c_{rm}\left[ x_k , x_r \right] ^e}
		&=\beta(x_m)+\sum_{
			\begin{subarray}{c} 
				I\subset W,I\neq \emptyset\\ 
				J\subset L\setminus W
			\end{subarray}	
		}(-1) ^{|I|+|J|}[x_k, \gcd(Y_{I\cup J})]^e \nonumber\\
		&=\beta(x_m)+\sum_{i=1}^w\sum_{I\subset \{i+1,\ldots,l\}}(-1) ^{|I|+1}[x_k,\gcd(Y_{I\cup \{i\}})]^e.
	\end{align}
	For each  integer $1\le i\le w$, we denote the index set $$I(i)=\{i, w+1,\ldots, a\}.$$ Then $I(i)=\{i\}\cup (A\setminus W)$ and $|I(i)|=b$. So by the definition of $b$ and (\ref{e3.23}), we get   $$\gcd(Y_{I(i)\cup (L\setminus A) })=\gcd(Y_L).$$
	It then follows from   $x_k=\gcd(Y_{L\setminus W})$ that 
	\begin{align}\nonumber
		(x_k,y_i)&=(\gcd(Y_{L\setminus W}),y_i)=\gcd(Y_{\{i\}\cup L\setminus W})\\
		&=\gcd(Y_{\{i\}\cup(A\setminus W)\cup (L\setminus A)})=\gcd(Y_{I(i)\cup (L\setminus A)})
		\nonumber\\
		&=\gcd(Y_L).\label{3.38}
	\end{align}
	Moreover, for any index set $I\subset \{i+1,\ldots,l\}$, since $x_k=\gcd(Y_{L\setminus W})$,   we  obtain by (\ref{3.38}) that
	\begin{align}\label{e3.20}
		\big(x_k,\gcd(Y_{I\cup \{i\}})\big)&=\big(\gcd(Y_{L\setminus W}),\gcd(Y_{I\cup \{i\}}) \big)=\gcd(Y_{ I\cup \{i\}\cup (L\setminus W) })\nonumber\\
		&=\big(\gcd(Y_{I}),\gcd(Y_{\{i\}\cup (L\setminus W)}) \big)=\big(\gcd(Y_{I}),\gcd(Y_{L}) \big)
		\nonumber\\
		&=\gcd(Y_L).
	\end{align}  
	Then by (\ref{e3.18})-(\ref{e3.20}), we deduce that
	\begin{align}\nonumber
		&\alpha _{e, S}(x_m)-\sum_{x_r|x_m}{c_{rm}\left[ x_k , x_r \right] ^e}\\
		=&\sum_{i=1}^w\sum_{I\subset \{i+1,\ldots,l\}}(-1) ^{|I|}\Big([x_k,\gcd(Y_{I\cup \{i\}})]^e-\gcd(Y_{I\cup \{i\}})^e\Big)\nonumber\\
		=&\sum_{i=1}^w\bigg(\big([x_k,y_i]^e-y_i^e\big)+\sum_{I\subset \{i+1,\ldots,l\},I\neq \emptyset}(-1) ^{|I|}
		\Big([x_k,\gcd(Y_{I\cup \{i\}})]^e-\gcd(Y_{I\cup \{i\}})^e
		\Big)\bigg)\nonumber\\
		=&\sum_{i=1}^w\Bigg(\Big(\frac{x_k^e}{\gcd(Y_L)^e}-1\Big)y_i^e+\sum_{I\subset \{i+1,\ldots,l\},I\neq \emptyset }(-1) ^{|I|}
		\Big(\frac{x_k^e}{\gcd(Y_L)^e}-1\Big)\gcd(Y_{I\cup \{i\}})^e
		\Big)\Bigg)
		\nonumber\\
		=&\bigg(\frac{x_k^e}{\gcd(Y_L)^e}-1\bigg)\sum_{i=1}^w
		\bigg(
		y_i^e+\sum_{I\subset \{i+1,\ldots,l\},I\neq \emptyset}(-1)^{|I|}\gcd(Y_{I\cup \{i\}})^e
		\bigg)\nonumber\\
		=&\bigg(\frac{x_k^e}{\gcd(Y_L)^e}-1\bigg)\sum_{i=1}^w
		\bigg(	y_i^e+\sum_{I\subset \{i+1,\ldots,l\},I\neq \emptyset}(-1)^{|I|}\gcd
		\Big(\bigcup_{j\in I}\{(y_i,y_j)^e\}\bigg).
		\label{e3.21}
	\end{align}
	Since $G_S(x_m)=\{y_1,\ldots,y_l\}$, we have $(y_i,y_{i+1}),\ldots,(y_i,y_{l})$ are proper divisors of $y_i$   for any given integer $1\le i\le w$. Then by Lemma \ref{lem8}, we deduce that for each integer $1\le i\le w$, 
	\begin{align}	\label{3.35e}	y_i^e+\sum_{I\subset \{i+1,\ldots,l\},I\neq \emptyset}(-1)^{|I|}\gcd
		\Big(\bigcup_{j\in I}\{(y_i,y_j)^e\}\Big)>0.
	\end{align}
	Since $x_k>\gcd(Y_L)$, we obtain from (\ref{e3.21}) and (\ref{3.35e})
	that
	\begin{align}
		\alpha _{e, S}(x_m)-\sum_{x_r|x_m}{c_{rm}\left[ x_k , x_r \right] ^e}>0\label{e3.22}
	\end{align}
	Hence by (\ref{e3.17}) and (\ref{e3.22}) we get 
	$$0<g\left( k,m \right)=\frac{1}{\alpha _{e,S}\left( x_m \right)}\sum_{x_r|x_m}{c_{rm}\left[x_k,x_r \right] ^e}<1,$$
	which means that $g(k,m)\not\in \mathbb{Z}$ as required. 
	
	This completes the proof of Lemma \ref{lem9} for  Case 1.
	
	{\sc Case 2.} {\it $\left[y_i, y_j\right]=x_m$  for all integers $1\leq i<j\leq l$, but there exists two integers $i$ and  $j$ with $1\leq i <j\leq l$ such that $\left(y_i, y_j\right)\notin G_s\left( y_i \right) \cap G_s\left( y_j \right) $.}
	
	Without loss of generality, we let $\left(y_1, y_2\right)\notin G_s\left(y_2\right)$. Then by the definition of greatest-type divisor, we know that the set 
	$$U:=\left\{u\in S:\left(y_1, y_2\right)|u|y_2, u\ne\left(y_1, y_2\right), u\ne y_2\right\}$$ 
	is nonempty.    Let $x_c$ be the smallest element of $U$. Since $S$ is a gcd-closed set, there exists an unique integer $d$ such that 
	$$x_d:=\left(x_c, \left(y_3, \ldots, y_l\right)\right)=(x_c,\gcd(Y_{L\setminus\{1,2\}})).$$  In what follows, we will show that $g(d,m)\not\in \mathbb{Z}$.
	
	We first  prove that $$[y_1,x_c]=[y_1,x_d]<x_m.$$ On the one hand, we suppose that $[y_1,x_c]=x_m$.  Note that $(y_1,(y_1,y_2))|\left(y_1, x_c\right)|\left(y_1,y_2\right)$ since $\left(y_1,y_2\right)|x_c|y_2$.  
	That is $\left(y_1,x_c\right)=\left(y_1,y_2\right)$. It then follows from  $[y_1,y_2]=x_m$ that 
	$$x_c=\frac{[y_1,x_c](y_1,x_c)}{y_1}=\frac{[y_1,y_2](y_1,y_2)}{y_1}=y_2,$$
	which contradicts to the fact $x_c\neq y_2$ since $x_c\in U$. Hence $[y_1,x_c]<x_m$. 
	On the other hand, noticing that  $\left[y_1,y_j\right]=x_m$ for all integers $2\le j\le l$, we get  $$\big[y_1,\gcd(Y_{L\setminus\{1,2\}})\big]
	=
	\gcd\bigg(\bigcup_{j=3}^l\{[y_1,y_j]\}\bigg)=x_m.$$ 
	Using the fact that $x_c\mid y_2$,  we obtain that 
	\begin{align*}
		\left[y_1,x_d\right]=\big[y_1,(x_c,\gcd(Y_{L\setminus\{1,2\}}))
		\big]
		=\big(\left[y_1,x_c\right],\big[y_1,\gcd(Y_{L\setminus\{1,2\}})\big]\big)
		=\left[y_1,x_c\right]
	\end{align*}
	as desired.

	Next, we will show that $(y_1,x_d)=\gcd(Y_L)$ and $x_d>\gcd(Y_L)$.  Since $x_c\in U$, we have $(y_1,y_2)\mid x_c\mid y_2$. Then by the definition of $x_d$, we get $\gcd(Y_L)\mid x_d\mid \gcd(Y_{L\setminus\{1\}})$.
	So $(y_1,\gcd(Y_L))|(y_1,x_d)|\gcd(Y_L),$ which means that  $$(y_1,x_d)=\gcd(Y_L).$$
	Since $x_c\in U$, we have $x_c\nmid y_1$. Otherwise, if $x_c\mid y_1$, then  we have $x_c\mid (y_1,y_2)$ since $x_c\mid y_2$. This contradicts to  $x_c\neq (y_1,y_2)$. Using the fact that $x_c\nmid y_1$, we get $[y_1,x_c]>y_1$. Then by $(y_1,x_d)=\gcd(Y_L)$ and $[y_1,x_c]=[y_1,x_d]$, we derive that 
	$$
	x_d=\frac{[y_1, x_d]}{y_1}(y_1,x_d)=\frac{[y_1, x_c]}{y_1}\gcd(Y_L)>\gcd(Y_L)
	$$
	as required. 
	
	Secondly,  noticing that 
	$S$ is a gcd-closed set and  $G_S(x_m)=Y_L$,  we obtain by Lemma \ref{lem6} that 
	\begin{align}
		\sum_{x_r|x_m}{c_{rm}\left[ x_d,x_r \right] ^e}&=\left[ x_d,x_m \right] ^e+\sum_{I\subset L,I\neq \emptyset}( -1 ) ^{|I|}\left[ x_d,\gcd(Y_I)\right]^e\label{3.35}\\
		&=x_m^e+\sum_{I\subset L,I\neq \emptyset}( -1 ) ^{|I|}\gcd\bigg(
		\bigcup_{i\in I}\{[x_d,y_i]\}\bigg)^e.	\label{e3.30}
	\end{align}
	Clearly,  $[y_i,x_d]=y_i<x_m$ since $x_d\mid y_i$ for each integer $i$ with $2\le i\le l$. And noticing that $[y_1,x_d]<x_m$, we deduce that 
	$[y_1,x_d],[y_2,x_d],\ldots,[y_l,x_d]$ are proper divisors of $x_m$. Then by Lemma \ref{lem8} and (\ref{e3.30}), we obtain that 
	\begin{align}
		\sum_{x_r|x_m}{c_{rm}\left[ x_d,x_r \right]^e}>0. \label{e3.31}
	\end{align}
	
	At the same time, since  $(y_1,x_d)=\gcd(Y_L)$, $x_d=(x_c,\gcd(Y_{L\setminus\{1,2\}}))$ and $x_d\mid y_2$, we deduce by (\ref{3.35}) that 
	\begin{align}
		\sum_{x_r|x_m}c_{rm}\left[x_d,x_r\right]^e
		&=\gamma(x_m)+\sum_{I\subset L\setminus\{1\}}(-1)^{|I|+1}\left[ x_d,(y_1,\gcd(Y_I))\right]^e \nonumber	\\
		&=\gamma(x_m)-\left[x_d,y_1\right]^e+\sum_{I\subset L\setminus\{1\}, I\neq \emptyset}(-1)^{|I|+1}\left[ x_d,(y_1,\gcd(Y_I))\right]^e\nonumber\\
		&=\gamma(x_m)-\frac{x_d^e~y_1^e}{\gcd(Y_L)^e}+\sum_{I\subset L\setminus\{1\}, I\neq \emptyset}(-1)^{|I|+1}
		\frac{x_d^e	(y_1,\gcd(Y_I))^e}{\gcd(Y_L)^e}\nonumber\\
		&=\gamma(x_m)-\frac{x_d^e}{\gcd(Y_L)^e}
		\bigg(y_1^e+\sum_{I\subset L\setminus\{1\}, I\neq\emptyset}(-1)^{|I|+1}(y_1,\gcd(Y_I))^e\bigg),\label{e3.32}
	\end{align}
	where $$\gamma(x_m):=x_m^e+\sum_{I\subset L\setminus\{1\}, I\neq \emptyset}(-1)^{|I|}\left[ x_d,\gcd(Y_I)\right]^e=x_m^e+\sum_{I\subset L\setminus\{1\}, I\neq \emptyset}(-1)^{|I|}\gcd(Y_I)^e.$$
	By Lemma \ref{lem2}, we have 
	\begin{align}\nonumber
		\alpha _{e, S}\left( x_m \right)&=x_{m}^{e}+\sum_{I\subset L,I\neq \emptyset}(-1) ^{|I|}\gcd(Y_I)^e\\
		&=\gamma(x_m)+
		\sum_{I\subset L\setminus\{1\}}(-1)^{|I|+1}(y_1,\gcd(Y_I))^e\nonumber\\
		&=\gamma(x_m)-y_1^e+
		\sum_{I\subset L\setminus\{1\}, I\neq \emptyset}(-1)^{|I|+1}(y_1,\gcd(Y_I))^e.
		\label{e3.33}
	\end{align}
	
	Finally, by (\ref{e3.32}) and  (\ref{e3.33}), we deduce that  
	\begin{align}\nonumber
		&\alpha _{e, S}(x_m)-\sum_{x_r|x_m}{c_{rm}\left[ x_d , x_r \right] ^e}\\
		=&\bigg(\frac{x_d^e}{\gcd(Y_L)^e}-1\bigg)
		\bigg(y_1^e+\sum_{I\subset L\setminus\{1\}, I\neq\emptyset}(-1)^{|I|}(y_1,\gcd(Y_I))^e\bigg)
		\nonumber\\
		&=\bigg(\frac{x_d^e}{\gcd(Y_L)^e}-1\bigg)
		\bigg(y_1^e+\sum_{I\subset L\setminus\{1\}, I\neq\emptyset}(-1)^{|I|}\gcd\bigg(\bigcup_{j\in I}\{(y_1,y_j)\}\bigg)^e\bigg).\label{e3.34}
	\end{align}
	For any given integers $2\le i\le l$, since $G_S(x_m)=Y_L$,   we have
	$(y_1,y_2),(y_1,y_3)\ldots,(y_1,y_l)$ are proper divisors of $y_1$. Using the fact that  $x_d>\gcd(Y_L)$, we obtain by (\ref{e3.34}) and Lemma \ref{lem8} that 
	\begin{align}
		\alpha _{e, S}(x_m)-\sum_{x_r|x_m}{c_{rm}\left[ x_d , x_r \right] ^e}
		>0.\label{e3.35}
	\end{align}
	Thus by (\ref{e3.31}) and (\ref{e3.35}), we get $0<g(d,m)<1$. This implies that $g(d,m)\not\in \mathbb{Z}$ as desired. So Case 2 is proved. 
	
	This finishes the proof of Lemma \ref{lem9}. 
	\hfill\ensuremath{\mbox{$\Box$}}\\

	
	\noindent{\bf Proof of the necessity.} Suppose that
	the power LCM matrix $\left[S^e\right]$ is divisible by the power GCD matrix $\left(S^e\right)$ in the ring $M_n\left(\mathbb{Z}\right)$. That is $[S^e](S^e)^{-1}\in M_n\left(\mathbb{Z}\right) $ or $(S^e)^{-1}[S^e]\in M_n\left(\mathbb{Z}\right)$. Note that
	$$\big([S^e](S^e)^{-1}\big)^\top=\big((S^e\big)^{-1}\big)^\top[S^e]^\top=\big((S^e\big)^\top\big)^{-1}[S^e]=(S^e)^{-1}[S^e]
	$$
	since $[S^e]$ and $(S^e)$ are symmetric integer matrices. It is sufficient to prove that if $S$ does not satisfy the condition $\mathcal{C}$, then $ [S^e](S^e)^{-1}\not\in M_n\left(\mathbb{Z}\right).$
	
	Now we assume that $S$ does not satisfy the condition $\mathcal{C}$.  In order to prove 
	necessity, in what follows,  we only need to show that 
	there exist two integers $1\le l,s\le n$ such that $$\Big([S^e](S^e)^{-1}\Big)_{ls}\not\in \mathbb{Z}.$$ 
	
	Since  $S$ does not satisfy the condition $\mathcal{C}$ and $\max_{x \in S} \{|G_{S}(x)|\} \geq 4$,  there at least exists  an element $x\in S$ such that $|G_S(x)|\ge2$ and not satisfying the condition $\mathcal{C}$.
	Let $x_m\in S$ be the largest element which  does not satisfy the condition $\mathcal{C}$.	
	Using the fact that  $S$ is a gcd-closed set, 
 we derive from  Lemmas \ref{lem2.9} and \ref{lem9} that there exists an element $x_l\in S$ such that  $g(l,m)\not\in \mathbb{Z}$. Then the set $\{x_i\in S:x_m|x_i \text{ and } g\left(l, i\right)\notin\mathbb{Z}\}$ is nonempty. Write $$x_s:=\max\{x_i\in S:x_m|x_i \text{ and } g\left(l, i\right)\notin\mathbb{Z}\}.$$
	So  $g\left(l, s\right)\notin\mathbb{Z}$ and $g\left(l, j\right)\in\mathbb{Z}$ if  $x_s|x_j$ and $x_s<x_j$. Then by Lemma \ref{lem1} and (\ref{1}), we deduce that 
	\begin{align}
		\left( \left[ S^e \right] \left( S^e \right) ^{-1} \right) _{ls}
		&=\sum_{r=1}^n{\left[ x_l, x_r \right] ^e}
		\sum_{\begin{subarray}{c} x_r\mid x_k \\ x_s\mid x_k
		\end{subarray}}{\frac{c_{rk}c_{sk}}{\alpha _{e, S}\left( x_k \right)}}\nonumber\\
		&=\sum_{x_s|x_k}c_{sk}\frac{1}{\alpha _{e, S}\left( x_k \right)}\sum_{x_r|x_k}c_{rs}\left[ x_l, x_r \right] ^e\nonumber\\
		&=\sum_{x_s|x_k}c_{sk}\cdot g(l,k)\nonumber\\
		&=g\left( l, s \right) +\sum_{\begin{subarray}{c}
				x_s|x_k\\
				x_s<x_k\\
		\end{subarray}}{c_{sk}\cdot g\left( l, k \right)}\notin\mathbb{Z}\nonumber
	\end{align}
	as required. 
	
	The proof of the necessity  of Theorem \ref{thm1} is complete.       \hfill\ensuremath{\mbox{$\Box$}}
	
	\section{Remarks}
	
	Clearly, Theorem \ref{thm1} confirms the Conjecture \ref{C3}. Consequently  Conjectures \ref{C1} and \ref{C2} are true. Furthermore, Theorem \ref{thm1} together with  the main results in \cite{FHZ09DM} and \cite{ZCH22JCTA} would give a complete answer to the Problem \ref{P1}.
	
	It is well-known that if the set $S$ is gcd-closed, then the structure $(S,\mid)$ itself constitutes a meet-semilattice. The nonsingularity of power GCD and power LCM matrices have been studied by using lattice-theoretic methods, see e.g. \cite{AST05LMA, HMM20JCTA,IK18LAA,MHM15JCTA}. Naturally, we can also view the divisibility of power GCD and power LCM matrices via the lattice structure. By Lemma \ref{lem3.1}, if $S$ is a gcd-closed set satisfying condition $\mathcal{C}$, then $(\langle G_S(x)  \rangle \cup \{x\}, \mid) $ is a Boolean algebra for all $x \in S$ with $|G_S(x)| \ge 2$. Inspired by the characteristic of Boolean algebra, we utilize inclusion of subsets to study the divisibility structure of elements of gcd-closed set $S$. For example,  we let  $x\in S$ satisfying condition $\mathcal{C}$, $G_S(x)=\{y_1,\ldots,y_l\}$ with  $l=2,3,4$,  and $y_{i_1\ldots i_k}=(y_{i_1},\ldots,y_{i_k})$. The divisibility structure of the set $\langle G_S(x)  \rangle \cup \{x\}$ are clearly illustrated in Figures 1 and 2. 
	\begin{figure}[!htb]
		\setlength{\unitlength}{1.5mm}
		\begin{center}
			\begin{picture}(40,25)
				\put(20,0){\circle*{0.7}}
				\put(21,-1){\footnotesize{$y$}}
				\put(28,8){\circle*{0.7}}
				\put(12,8){\circle*{0.7}}
				\put(29,7){\footnotesize{$y_{1}$}}
				\put(9,7){\footnotesize{$y_{2}$}}
				
				\put(20,17){\footnotesize{$x$}}
				
				\put(20,16){\circle*{0.7}}

				\put(20,0){\line(1,1){8}}
				\put(20,0){\line(-1,1){8}}
				\put(12,8){\line(1,1){8}}
				\put(28,8){\line(-1,1){8}}
			\end{picture}
			\begin{picture}(40,25)
				\put(20,0){\circle*{0.7}}
				\put(21,-1){\footnotesize{$y$}}
				\put(20,0){\line(0,1){8}}	
				
				\put(20,8){\circle*{0.7}}
				\put(28,8){\circle*{0.7}}
				\put(12,8){\circle*{0.7}}
				\put(29,7){\footnotesize{$y_{12}$}}
				\put(21,7){\footnotesize{$y_{13}$}}
				\put(9,7){\footnotesize{$y_{23}$}}

				\put(20,16){\circle*{0.7}}
				\put(28,16){\circle*{0.7}}
				\put(12,16){\circle*{0.7}}
				\put(29,15){\footnotesize{$y_{1}$}}
				\put(22,15){\footnotesize{$y_{2}$}}
				\put(9,15){\footnotesize{$y_{3}$}}
				
				\put(20,24){\circle*{0.7}}
				\put(20,25){\footnotesize{$x$}}
				
				\put(20,0){\line(0,1){8}}
				\put(20,0){\line(-1,1){8}}
				\put(20,0){\line(1,1){8}}
				\put(12,8){\line(0,1){8}}
				\put(12,8){\line(1,1){8}}
				\put(28,8){\line(0,1){8}}
				\put(28,8){\line(-1,1){8}}
				\put(20,8){\line(-1,1){8}}
				\put(20,8){\line(1,1){8}}
				\put(20,16){\line(0,1){8}}
				\put(28,16){\line(-1,1){8}}
				\put(12,16){\line(1,1){8}}	
			\end{picture}
		\end{center}
		\vspace{-5mm}
		\caption{The Boolean lattice  $B_2$ and $B_3$.}\label{fig}
	\end{figure}
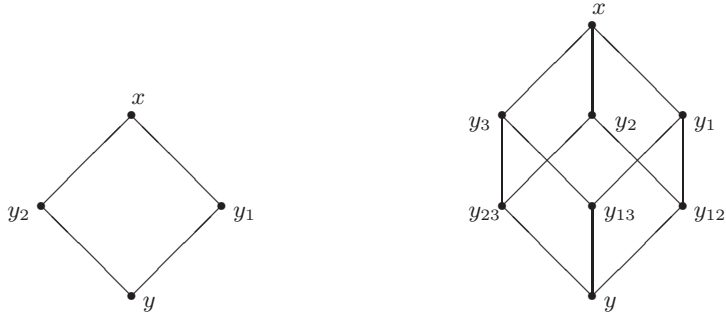
	\begin{figure}[!htb]
		\setlength{\unitlength}{1.5mm}
		\begin{center}
			\begin{picture}(40,30)
				\put(20,0){\circle*{0.7}}\put(20.5,-1.5){\footnotesize{$y$}}
				\put(20,0){\line(1,2){4}}
				\put(20,0){\line(3,2){12}}
				\put(20,0){\line(-1,2){4}}
				\put(20,0){\line(-3,2){12}}
				
				\put(32,8){\circle*{0.7}}
				\put(24,8){\circle*{0.7}}
				\put(16,8){\circle*{0.7}}
				\put(8,8){\circle*{0.7}}
				\put(33,7.5){\footnotesize{$y_{123}$}}
				\put(24.5,7){\footnotesize{$y_{124}$}}
				\put(17.5,7.5){\footnotesize{$y_{134}$}}
				\put(3.5,7.5){\footnotesize{$y_{234}$}}

				\put(32,8){\line(1,1){8}}
				\put(32,8){\line(-1,1){8}}
				\put(32,8){\line(0,1){8}}
				\put(24,8){\line(2,1){16}}
				\put(24,8){\line(-1,1){8}}
				\put(24,8){\line(-2,1){16}}
				\put(16,8){\line(0,1){8}}
				\put(16,8){\line(2,1){16}}
				\put(16,8){\line(-2,1){16}}
				\put(8,8){\line(0,1){8}}
				\put(8,8){\line(2,1){16}}
				\put(8,8){\line(-1,1){8}}

				\put(40,16){\circle*{0.7}}
				\put(32,16){\circle*{0.7}}
				\put(24,16){\circle*{0.7}}
				\put(16,16){\circle*{0.7}}
				\put(8,16){\circle*{0.7}}
				\put(0,16){\circle*{0.7}}
				\put(41,15.5){\footnotesize{$y_{12}$}}
				\put(33,15.5){\footnotesize{$y_{13}$}}
				\put(25,15.5){\footnotesize{$y_{23}$}}
				\put(17,15.5){\footnotesize{$y_{14}$}}
				\put(4.5,15.5){\footnotesize{$y_{24}$}}\put(-4,15.5){\footnotesize{$y_{34}$}}

				\put(40,16){\line(-1,1){8}}
				\put(40,16){\line(-2,1){16}}
				\put(32,16){\line(-2,1){16}}
				\put(32,16){\line(0,1){8}}
				\put(24,16){\line(-1,1){8}}
				\put(24,16){\line(0,1){8}}
				\put(16,16){\line(2,1){16}}
				\put(16,16){\line(-1,1){8}}
				\put(8,16){\line(2,1){16}}
				\put(8,16){\line(0,1){8}}
				\put(0,16){\line(2,1){16}}
				\put(0,16){\line(1,1){8}}
				
				\put(32,24){\circle*{0.7}}
				\put(24,24){\circle*{0.7}}
				\put(16,24){\circle*{0.7}}
				\put(8,24){\circle*{0.7}}
				\put(32.5,24){\footnotesize{$y_1$}}
				\put(24.5,24){\footnotesize{$y_2$}}
				\put(17,24){\footnotesize{$y_3$}}
				\put(5,24){\footnotesize{$y_4$}}
				\put(32,24){\line(-3,2){12}}
				\put(24,24){\line(-1,2){4}}
				\put(16,24){\line(1,2){4}}
				\put(8,24){\line(3,2){12}}
				\put(20,32){\circle*{0.7}}
				\put(19.5,33){\footnotesize{$x$}}
				
			\end{picture}
		\end{center}
		\vspace{-5mm}
		\caption{The Boolean lattice  $B_4$.}\label{fig}
	\end{figure}
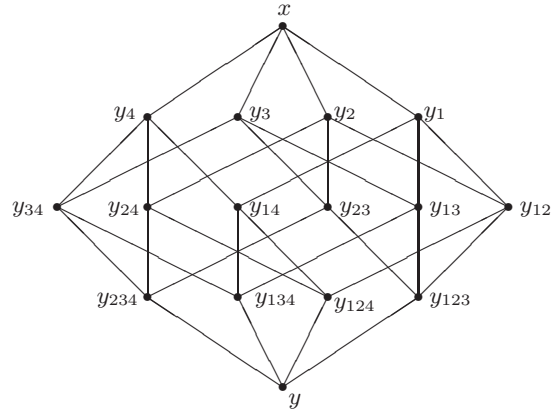

	Some scholars are also interested in the divisibility of GCD-type matrix and LCM-type matrix associated with an arithmetic function, see, e.g.  \cite{BL95LAA,TLC16LMA,Z23CM}.  
	For any given arithmetic function $f$, denote by $ (f(S)):= (f(x_i,x_j)) $ the $ n \times n $ matrix having $ f $ evaluated at the greatest common divisor $(x_i, x_j)$ of $ x_i $ and $ x_j $ as its $ ij\raisebox{0mm}{-}\text{entry} $ and $  (f[S]):= (f[x_i,x_j]) $ the $ n \times n $  matrix having $ f $ evaluated at the least common multiple $ [x_i,x_j] $ of $ x_i $ and $ x_j $ as its 
	$ ij\raisebox{0mm}{-}\text{entry} $. Given any set $S$ of positive integers, denote two classes of arithmetic functions
	\begin{equation*}
		\mathcal{C}_S:=\{f :(f * \mu)(d) \in \mathbb{Z} \text { whenever } d \mid \operatorname{lcm}(S)\}.
	\end{equation*}
	and 
	$$
	\mathcal{D}_S := \{ f \in \mathcal{C}_S: \ f(x) \mid f(y) \ \text{whenever} \ x \mid y \ \text{and} \ x,y \in S \}.
	$$ 
	For any given arithmetic function $f\in \mathcal{C}_S$, Bourque and Ligh \cite{BL93LMA} proved that  $(f(S))$ and $(f[S])$ are integer matrices. In 2011, Li and Tan \cite{LT11DM} showed that if $ S $ is a gcd-closed set with $ \max_{x \in S} \{|G_{S}(x)|\} = 1 $, $ f \in \mathcal{D}_S $ is multiplicative with  $ (f(S)) $ being nonsingular, then $ (f(S))\mid (f[S]) $. There is currently relatively little research on this topic.
	To enclose this paper, we propose the following open problem.
	
	\begin{problem}\label{P2}
		Let $ S $ be a gcd-closed set and  $ f \in \mathcal{C}_S$ be a multiplicative function with the matrix $ (f(S)) $ being nonsingular. Characterize the gcd-closed set $S$  and the function $ f $ such that the matrix $ (f(S)) $ divides the matrix $ (f[S]) $ in the ring $ M_{|S|}(\mathbb{Z}) $.
	\end{problem}
	
	Evidently, Theorem \ref{thm1}, as well as the main results in \cite{FHZ09DM} and \cite{ZCH22JCTA} answer Problem \ref{P2} for the case when $f$ is a power arithmetic function.

\end{document}